\documentclass{amsart}
\usepackage{amssymb}

\theoremstyle{plain}
\newtheorem{theorem}{Theorem}
\newtheorem{proposition}{Proposition}
\newtheorem{lemma}[proposition]{Lemma}

\theoremstyle{definition}
\newtheorem{definition}{Definition}
\newtheorem{example}{Example}
\newtheorem{remark}{Remark}
\newtheorem*{question}{Question}

\begin{document}

\title{Finite self-similar $p$-groups with abelian first level stabilizers}

\author{Zoran {\v S}uni\'c}
\address{Department of Mathematics, MS-3368, Texas A\&M University, College Station, TX 77843-3368, USA}
\email{sunic@math.tamu.edu}
\thanks{Partially supported by the NSF under DMS-0805932.}

\begin{abstract}
We determine all finite $p$-groups that admit a faithful, self-similar action on the $p$-ary rooted tree such that the first level stabilizer is abelian. A group is in this class if and only if it is a split extension of an elementary abelian $p$-group by a cyclic group of order $p$. 

The proof is based on use of virtual endomorphisms. In this context the result says that if $G$ is a finite $p$-group with abelian subgroup $H$ of index $p$, then there exists a virtual endomorphism of $G$ with trivial core and domain $H$ if and only if $G$ is a split extension of $H$ and $H$ is an elementary abelian $p$-group. 
\end{abstract}



\maketitle

%------------------------------------------------------------------------

\section*{Introduction}

The goal of this paper is to prove the following result. 

\begin{theorem} \label{t:main}
A finite $p$-group admits a faithful, self-similar action on the rooted $p$-ary tree such that the first level stabilizer is abelian if and only if it is a split extension of an elementary abelian $p$-group by a cyclic group of order $p$. 
\end{theorem}

The result can be recast in a purely algebraic form using the language of virtual endomorphisms (explained in Section~\ref{s:virtual}), and it is this reformulated statement that we actually prove in Section~\ref{s:proof}. 

\begin{theorem} \label{t:virtual}
Let $G$ be a group that fits in a short exact sequence  
\[ 1 \to H \to G \to C_p \to 1, \]
where $H$ is a finite, nontrivial, abelian $p$-group and $C_p$ denotes the cyclic group of order $p$. 

There exists a virtual endomorphism $\phi:H \to_p G$ with trivial core if and only if the sequence splits and $H$ is an elementary abelian $p$-group. 
\end{theorem}

The world of groups that admit faithful, self-similar actions on regular rooted trees is rich in important and colorful examples, such as the first Grigorchuk group~\cite{grigorchuk:burnside} and groups related to it~\cite{grigorchuk:gdegree,grigorchuk:pgps,bartholdi-s:wpg,sunic:hausdorff}, Gupta-Sidki $p$-groups~\cite{gupta-s:burnside, gupta-s:p}, Basilica group~\cite{grigorchuk-z:basilica1,bartholdi-v:basilica}, Hanoi Towers groups~\cite{grigorchuk-s:hanoi-cr,grigorchuk-s:hanoi-spectrum} and the plethora of examples of groups obtained as iterated monodromy groups of self-coverings of the Riemann sphere by post-critically finite rational maps~\cite{nekrashevych:b-selfsimilar,bartholdi-n:rabbit} (and some more general finite partial self-coverings of topological spaces and orbispaces~\cite{nekrashevych:b-selfsimilar}).  

A priori, there are seemingly few necessary conditions for a group to admit a faithful, self-simiar action on a regular rooted tree. For instance, the group needs to be residually finite and there should be only finitely many different primes dividing the orders of the elements of finite order in the group.  

A straightforward way to confirm that some well known group $G$ admits a faithful, self-similar action on a rooted tree is by constructing a self-similar group and then proving that the constructed group is isomorphic to $G$. For instance, this was the way a faithful, self-similar action was established for the solvable Baumslag-Solitar groups and other ascending HNN extensions of free abelian groups~\cite{bartholdi-s:bs}, as well as for free groups of finite rank~\cite{vorobets-v:free,vorobets-v:free2,steinberg-v-v:free} and free products of finitely many copies of $C_2$~\cite{savchuk-v:free-c2}.   

A more organized approach, which in theory could work for any group, is provided by use of virtual endomorphisms. This approach is used in~\cite{nekrashevych-s:half} to construct actions of free abelian groups on the binary tree, in~\cite{berlatto-s:virtual} to construct actions of some torsion-free, nilpotent groups, and in~\cite{kapovich:arithmetic-selfsimilar} to show that an irreducible lattice in a semisimple algebraic group is virtually isomorphic to an arithmetic lattice if and only if it admits a faithful, self-similar action on a rooted tree. 

While a structural description of the class of self-similar groups that  is not a mere restatement of the definition involving the action on a tree seems well beyond reach, there is more hope in some restricted settings. Nekrashevych and Sidki showed that if a finitely generated, nilpotent group admits a faithful, self-similar action on the binary tree, then it is either free abelian group of finite rank or a finite 2-group~\cite{nekrashevych-s:half}. They proceeded by classifying the faithful, self-similar actions of free abelian groups of finite rank on the binary rooted tree. Our study of finite, self-similar $p$-groups is complementary to their work as well as to the study of self-similar, torsion-free, finitely generated, nilpotent groups in~\cite{berlatto-s:virtual} and self-similar, abelian groups in~\cite{brunner-s:abelian}.  In this context one may ask the following. 

\begin{question}
Which finite $p$-groups admit a faithful, self-similar action on the rooted $p$-ary tree?
\end{question}

In addition to the interpretation in terms of virtual endomorphisms, the question can also be restated in terms of finite automata (in this setting, self-similar groups are often called state-closed groups~\cite{sidki:circuit} or automaton groups~\cite{grigorchuk-n-s:automata}), but we will not pursue this aspect here.  

Note that the self-similarity is essential, since every finite $p$-group admits a faithful action on the rooted $p$-ary tree. Indeed, every finite $p$-group embeds, for sufficiently large $n$, in the Sylow $p$-subgroup of the symmetric group $S_{p^n}$. The Sylow $p$-subgroup of $S_{p^n}$ is the iterated wreath product $K_{p,n}=C_p \wr (C_p \wr ( C_p \wr \dots \wr C_p))$, where there are $n$ factors (see, for instance,~\cite{kaloujnine:p-sylow}; note that, in our notation, in each step of the iteration the cyclic factor on the left is the active one). On the other hand, $K_{p,n}$ admits a faithful, self-similar action on the rooted $p$-ary tree. Therefore, every finite $p$-group embeds in a finite, self-silmiar $p$-group. 

%--------------------------------------------------

\subsection*{Acknowledgments} The author is thankful to Rostislav Grigorchuk, Said Sidki, Volodymyr Nekrashevych, Marcin Mazur and the especially the referee for their helpful remarks. 

%------------------------------------------------------------------

\section{Virtual endomorphisms} \label{s:virtual}

This section contains the necessary background on virtual endomorphisms. 

\begin{definition}
A homomorphism $\phi: H \to G$ from a subgroup $H$ of finite index $k$ in $G$ to $G$ is called \emph{virtual endomorphism}. Such a homomorphism is also called $\frac1k$-endomorphism of $G$ and denoted by $\phi: H \to_k G$.  
\end{definition}

Every virtual endomorphism $\phi: H \to_k G$ induces a self-similar action of $G$ on a  rooted $k$-ary tree by automorphisms. We postpone the description of the action, but we point out that it is not always faithful. The following is proved in~\cite[Proposition 3.2]{{nekrashevych:virtual}} (see also~\cite[Proposition 2.7.5]{nekrashevych:b-selfsimilar}). 

\begin{proposition} \label{p:faithful}
The action of $G$ on the $k$-ary rooted tree induced by the virtual endomorphism $\phi: H \to_k G$ is faithful if and only if the only normal subgroup of $G$ contained in $H$ that is $\phi$-invariant is the trivial group. 
\end{proposition}

Note that, by definition, $K$ is $\phi$-invariant if $\phi(K) \subseteq(K)$. 

The faithfulness criterion motivates the following definition. 

\begin{definition}
The \emph{core} of a virtual endomorphism $\phi: H \to_k G$ is the largest normal subgroup of $G$ contained in $H$ that is $\phi$-invariant. 

A virtual endomorphism $\phi :H \to_k G$ is \emph{simple} if it has trivial core. 
\end{definition}

Note that the core is well defined since if $N_1$ and $N_2$ are two normal subgroups of $G$ contained in $H$ that are $\phi$-invariant, then so is their product $N_1N_2$. 

%------------------------------------------------------------------------

\section{Virtual $\frac1p$-endomorphisms}

Before we move on to the case of virtual $\frac1p$-endomorphisms of finite $p$-groups with abelian domain we provide several examples and general results that will be of use in the subsequent sections. 

\begin{example} \label{e:elementary}
We provide an injective, simple virtual endomorphism $\phi: H \to_p G$ of the elementary abelian $p$-group $G=C_p^m$. 

Let $S=\{e_1,\dots,e_m\}$ be a generating set and $H=\langle e_2,\dots,e_m\rangle$. Define an injective virtual endomorphism $\phi:H \to_p G$ by $\phi(e_2^{\epsilon_2}\dots e_m^{\epsilon_m})=e_1^{\epsilon_2}\dots e_{m-1}^{\epsilon_m}$, for $\epsilon_2,\dots,\epsilon_m \in \{0,\dots,p-1\}$. 

Clearly, the only $\phi$-invariant subgroup of $H$ is the trivial subgroup. Therefore the core of $\phi$ is trivial. 
\end{example}

\begin{proposition} \label{p:K}
Let $\phi: H \to _p G$ be a simple, non-injective virtual endomorphism, where $G$ is a finite, non-abelian $p$-group. Let $K$ be the kernel of $\phi$. Then

\textup{(i)} $K$ does not contain any nontrivial normal subgroup of $G$. 

\textup{(ii)} For every $t \in G \setminus H$ and $i=0,\dots,p-1$, $K^{t^i}$ is normal in $H$ and $K \cap K^t \cap \dots \cap K^{t^{p-1}} = 1$.   

\textup{(iii)} The group $H$ is a nontrivial subdirect product. 
\end{proposition}

\begin{proof}
(i) Every subgroup of $K$ is $\phi$-invariant. Since the core of $\phi$ is trivial, no nontrivial subgroup of $K$ is normal in $G$. 

(ii) Since $K$ is normal in $H$ and $H$ is normal in $G$, the conjugates of $K$ in $G$ are normal in $H$.  

Since $K$ is normal in $H$, but not normal in $G$, and $[G:H]=p$, the group $H$ is the normalizer of $K$ and there are exactly $p$ distinct conjugates of $K$ in $G$, namely $K,K^t,\dots,K^{t^{p-1}}$. Their intersection is the core of $K$ in $G$ (the largest subgroup of $K$ that is normal in $G$), which is trivial by (ii).  

(iii) This follows from the fact that $K,K^t,\dots,K^{t^{p-1}}$ are nontrivial, normal subgroups of $H$ with trivial intersection. 
\end{proof}

The above observations can be used to establish that some finite $p$-groups do not admit simple $\frac1p$-endomorphisms. 

\begin{example}[Dihedral groups]
Consider the dihedral group $D_{2^n}= \langle a,b \mid a^2=b^2=(ab)^{2^n}=1 \rangle$, for $n \geq 3$. 

Each of the maximal subgroups of $D_{2^n}$ is subdirectly irreducible. Therefore, by Proposition~\ref{p:K}, $D_{2^n}$ does not admit simple, non-injective $\frac12$-endomorphisms. On the other hand, the only two isomorphic maximal subgroups of $D_{2^n}$ share their Frattini subgroup and this subgroup is nontrivial, normal subgroup of $D_{2^n}$. Any injective virtual endomorphism would keep this group invariant. Therefore, $D_{2^n}$ does not admit simple, injective $\frac12$-endomorphisms either. 

\end{example}

We end this section with a fairly general necessary condition for the existence of simple $\frac1p$-endomorphisms. 

\begin{lemma}[Splitting Lemma] \label{l:splitting}
Let $G$ be a group that fits in a short exact sequence  
\[ 1 \to H \to G \to C_p \to 1.  \]

If $\phi:H \to_p G$ is a simple endomorphism and $H$ contains elements of order $p$, then $\phi(H) \setminus H$ also contains elements of order $p$. 

Moreover, the given sequence splits. 
\end{lemma}

\begin{proof}
Let $P$ be the set of elements in $H$ of order dividing $p$. The set $P$ contains nontrivial elements and the group $\langle P \rangle$ is a nontrivial, characteristic subgroup of $H$. Therefore, $\langle P \rangle$ is nontrivial, normal subgroup of $G$. The elements of $P$ are mapped under $\phi$ to elements of order dividing $p$. If all elements of $P$ are mapped inside $H$, then they are mapped to other elements in $P$. In that case $\langle P \rangle$ would be a nontrivial, normal, $\phi$-invariant subgroup of $G$, a contradiction. Therefore there exists an element in $P$ that is mapped under $\phi$ outside of $H$. However, the image of such an element has order $p$ (it cannot be trivial since it is outside of $H$). 

Since $G \setminus H$ contains elements of order $p$, the sequence splits. 
\end{proof}

%------------------------------------------------------------------------

\section{Virtual $\frac1p$-endomorphisms of finite $p$-groups with abelian domain}\label{s:proof}

One direction of Theorem~\ref{t:virtual} is a corollary of the following more general result on regular $p$-groups. Recall that a $p$-group $G$ is a regular $p$-group if, for every $a$ and $b$ in $G$ there exists $c$ in $[\langle a,b\rangle, \langle a,b \rangle]$ such that $(ab)^p = a^pb^pc^p$. 

\begin{proposition}\label{p:regular}
Let $G$ be a group that fits in a short exact sequence  
\[ 1 \to H \to G \to C_p \to 1, \]
where $H$ is a finite, nontrivial $p$-group. 

If there exists a simple virtual endomorphism $\phi:H \to_p G$ such that $\phi(H)$ is a regular $p$-group, then the sequence splits and $H$ has exponent $p$.  
\end{proposition}

\begin{proof}
The sequence splits by the Splitting Lemma. Further, let $a$ be an element of order $p$ in $\phi(H) \setminus H$. Such an element exists by the Splitting Lemma. Since $[G:H]=p$ and $\phi(H) \neq \phi(H) \cap H$, it follows that $[\phi(H):(\phi(H)\cap H)]=p$. Therefore, 
\begin{equation} \label{e:phiH}
 \phi(H) = \langle a \rangle \ltimes (\phi(H) \cap H).  
\end{equation} 

It follows from~\eqref{e:phiH} that, for $h \in H$, there exists $e \in \{0,\dots,p-1\}$ and $h_1 \in H \cap \phi(H)$ such that 
\[ \phi(h) = a^e h_1. \]
Since $\phi(H)$ is regular and both $a$ and $h_1$ are in $\phi(H)$, this implies that 
\begin{equation} \label{e:phihp}
 \phi(h^p) = (a^e h_1)^p = a^{ep} h_1^p c^p = h_1^p c^p, 
\end{equation}
where $c \in [\langle a, h_1 \rangle, \langle a, h_1 \rangle]$. However, $H$ is normal in $G$, showing that $c \in H$. 

Consider $H^p$, the subgroup of $H$ generated by the $p$th powers of the elements in $H$. This group is characteristic in $H$, hence normal in $G$, and~\eqref{e:phihp} shows that it is $\phi$-invariant. Since $\phi$ has trivial core, $H^p$ is trivial. Thus $H$ has exponent $p$.  
\end{proof}

\begin{proof}[Proof of the forward direction of Theorem~\ref{t:virtual}]
The sequence splits by the Splitting Lemma. Since $H$ is abelian it is regular, and so is its image $\phi(H)$. By Proposition~\ref{p:regular}, $H$ has exponent $p$. Therefore $H$ is elementary abelian $p$-group. 
\end{proof}

\begin{proof}[Proof of the backward direction of Theorem~\ref{t:virtual}]
Let $a$ be an element of order $p$ outside of $H$. Conjugation by $a$ induces an automorphism $\alpha$ on $H$. 

We may, when convenient, think of $H$ as the vector space $V$  of dimension $n$ over the field on $p$ elements and $\alpha$ as an element in $GL(n,p)$ of order $p$. 

Since $\alpha$ has order $p$, the minimal polynomial of $\alpha$ divides $x^p -1=(x-1)^p$. Therefore, 1 is the only eigenvalue and $\alpha$ admits a Jordan normal form in which every block is of size no greater than $p \times p$ and has the form 
\[ 
 \begin{pmatrix}
 1 & 1 & 0 & \dots & 0 & 0 \\
 0 & 1 & 1 & \dots & 0 & 0 \\
 0 & 0 & 1 & \dots & 0 & 0\\
 \dots \\
 0 & 0 & 0 & \dots & 1 & 1 \\    
 0 & 0 & 0 & \dots & 0 & 1    
 \end{pmatrix}.
\]
Let $s_1 \geq \dots \geq s_m$ be the sizes of the Jordan blocks of $\alpha$ and  
\[
 B = \{ \ 
 b_{1,1}, \  b_{1,2}, \ \dots \ , b_{1,s_1}, \ 
 b_{2,1}, \  b_{2,2}, \ \dots \ , b_{2,s_2}, \ 
 \dots , \
 b_{m,1}, \  b_{m,2}, \ \dots \ , b_{m,s_m} \ \}
\]  
be the basis corresponding to the Jordan form.  The eigenspace of $V$ corresponding to the eigenvalue 1 is $E_1 = \langle b_{1,1}, b_{2,1}, \dots, b_{m,1} \rangle$. Denote $K=\langle \{ b_{i,j} \mid i=1,\dots,m, \ j=2,\dots,s_i \} \rangle$. and note that $V = E_1 \oplus K$. 

Define $\phi: H \to_p G$ on $E_1$ by  
\begin{equation} \label{e:phi}
 b_{1,1} \mapsto  b_{2,1} \mapsto b_{3,1} \mapsto \dots \mapsto  b_{m-1,1} \mapsto b_{m,1} \mapsto a, 
\end{equation}
and 
\[ b_{i,j} \mapsto 1, \]
for the basis elements in $K$. 

Note that $(b_{i,1})^a=b_{i,1}$, for $i=1,\dots,m$, and therefore 
\[
 b_{2,1}, \  b_{3,1}, \ \dots \ , \ b_{m-1,1}, \ b_{m,1}, \ a,
\] 
commute and generate a subgroup of $G$ isomorphic to $C_p^m$. This implies that $\phi$ is a well defined homomorphism (the definition of $\phi$ given above on the generators of $H$ can be extended to the whole group) and $K=Ker(\phi)$. 

Any subgroup of $H$ that is normal in $G$ corresponds, in the vector space point of view, to a subspace of $V$ that is invariant under $\alpha$. Such subspaces contain nontrivial vectors that are fixed by $\alpha$. Therefore, every nontrivial subgroup of $H$ that is normal in $G$ contains nontrivial elements of $E_1$. On the other hand, the nontrivial elements of $E_1$ are pushed outside of $H$ by iterations of $\phi$ (see~\eqref{e:phi}). This implies that there are no $\phi$-invariant subgroups of $H$ that are normal in $G$. 

Therefore $\phi$ is a simple $\frac1p$-endomorphism of $G$. 
\end{proof}

\begin{remark}
We may extend the approach taken in the proof of the backward direction of Theorem~\ref{t:virtual} as follows.  

Let $G$ be a $p$-group that fits in a split short exact sequence  
\[ 1 \to H \to G \to C_p \to 1, \]
where $H$ is a finite, nontrivial $p$-group. Let $E_1$ be the intersection of $H$ and the maximal elementary abelian $p$-group in the center of $G$. Note that this group is 
nontrivial, since $H$ is normal in $G$ and every normal subgroup of a $p$-group intersects the center nontrivially. Every subgroup of $H$ that is normal in $G$ must contain an element of $E_1$ (the role of $E_1$ is exactly the same as the role of the eigenspace $E_1$ in the 
proof of Theorem~\ref{t:virtual}). Let $E_1 = \langle b_1, b_2, \dots, b_m \rangle$ and let $a$ be an element of order $p$ in $G \setminus H$. The map given by   
\[
 b_1 \mapsto b_2 \mapsto \dots \mapsto b_m \mapsto a
\]
can be extended to a homomorphism $\phi_0: E_1 \to G$, since $a$ has order $p$ and commutes with $b_1,b_2,\dots,b_m$. If this map can be further extended to a homomorphism $\phi:H \to G$ (for instance if there exists a normal subgroup $K$ of $H$ such that $H=E_1 \ltimes K$ as in the proof of Theorem~\ref{t:virtual}), then this homomorphism would be a simple virtual endomorphism $\phi:H \to_p G$. 
\end{remark}

%------------------------------------------------------------------------

\section{Virtual endomorphisms and self-similar actions}

In this section we present a brief description of the relation between self-similar actions and virtual endomorphisms, providing ``translation'' between the languages and settings of Theorem~\ref{t:main} and Theorem~\ref{t:virtual}. For original definitions and more details see~\cite{nekrashevych:virtual} or~\cite{nekrashevych:b-selfsimilar}). 

Let $X=\{0,\dots,k-1\}$. Define the rooted $k$-ary tree as the graph with vertex set $X^*$, the set of words over $X$, where each vertex $u$ has $k$ children, the vertices $ux$, for $x \in X$. The empty word is the root of the tree.  

Every automorphism $g$ of the tree $X^*$ decomposes as 
\[ g = \pi_g \ (g|_0,g|_1,\dots,g|_{k-1}), \]
where $\pi_g$ is a permutation of the alphabet $X$, called root permutation of $g$, describing the action of $g$ on the first level of the tree, and $g|_0,\dots,g|_{k-1}$ are tree automorphisms, called sections of $g$, describing the action of $g$ on the subtrees below the first level. For every letter $x$ in $X$ and word $w$ over $X$, 
\[ g(xw) = \pi_g(x)g|_x(w). \]

\begin{definition}
A group of tree automorphisms $G$ is \emph{self-similar} if all sections of all elements in $G$ are elements in $G$.  
\end{definition}

Let $G$ be a self-similar group of automorphisms of the $k$-ary rooted tree $X^*$ acting transitively on the first level of the tree and let $H$ be the stabilizer of the vertex 0. A simple virtual $\frac{1}{k}$-endomorphism $\phi:H \to_k G$ may be defined by 
\[
 \phi(h) = h|_0. 
\]

We now provide a brief description of the self-similar action on a rooted $k$-ary tree by automorphisms associated to a virtual endomorphism $\phi: H \to _k G$ (see~\cite[Proposition 4.9]{nekrashevych:virtual} or~\cite[Proposition 2.5.10]{nekrashevych:b-selfsimilar}). 

Choose a transversal $T=\{t_0,t_1,\dots,t_{k-1} \}$ for $H$ in $G$, with $t_0=1$. For $g \in G$, let $\overline{g}$ denote the representative of the left coset $gH$. 

A self-similar action of $G$ on the rooted $k$-ary tree induced by $\phi$ is defined as follows. For $g \in G$, define the root permutation $\pi_g$ of $X=\{0,1,\dots,k-1\}$ by 
\[ \pi_g(x) = y  \quad \text{if and only if} \quad \overline{gt_{x}} = t_y \]
and the section of $g$ at $x \in X$ by 
\[ g|_x = \phi(\overline{gt_x}^{-1}gt_x).\] 

The action induced by the virtual endomorphism $\phi$ may be not faithful, and Proposition~\ref{p:faithful} provides a necessary and sufficient condition on $\phi$ which ensures the  faithfulness of the action. 

Together, the constructions of a simple virtual endomorphism from a self-similar group and a faithful, self-similar action from a simple virtual endomorphism provide a bridge between Theorem~\ref{t:main} and Theorem~\ref{t:virtual}. 

It should probably be pointed out that in the general constructions $H$ always plays the role of a stabilizer of a vertex on level 1 in the tree, while we thought of it as the stabilizer of the first level. However, since our groups are $p$-groups acting nontrivially on the first level of a $p$-ary tree, there is no difference between the stabilizer of the first level and the stabilizer of a vertex on the first level. 

We end with an example illustrating the construction of a self-similar action on a regular rooted tree from a virtual endomorphism. The example follows the virtual endomorphism construction from the proof of Theorem~\ref{t:virtual}. 

\begin{example}
Let $H = C_3^4 = \langle b, c, d, e \rangle$ and 
\[
 G= \langle a \rangle \ltimes H  = \langle a, H \mid a^3=1, \ b^a = b, \ c^a=bc, \ d^a=d, \ e^a = e \rangle.  
\]
The automorphism of $H=C_3^4$ induced by conjugation by $a$ is given by the matrix 
\[ 
\alpha = 
\begin{pmatrix}
 1 & 1 & 0 & 0 \\
 0 & 1 & 0 & 0 \\
 0 & 0 & 1 & 0 \\
 0 & 0 & 0 & 1
\end{pmatrix}
\]
and $\phi$ is defined by $b \mapsto d \mapsto e \mapsto a$ and $c \mapsto 1$.  Take $t_0=1$, $t_1=a$, and $t_2=a^2$ as coset representatives of $H$ in $G$. The action of $G$ on the rooted ternary tree is defined by 
\[
\begin{array}{llllcr}
 a  &= &(012) &(1, &1,&1),  \\
 b  &= &         &(d, &d,&d), \\
 c  &= &         &(1, &d,&d^2), \\
 d  &= &         &(e,&e,&e), \\
 e  &= &         &(a,&a,&a).  
\end{array}
\]
\end{example}

%--------------------------------------------------

\def\cprime{$'$}

%\bibliographystyle{alpha}
%\bibliography{../smath}

\begin{thebibliography}{GNS00}

\bibitem[BN06]{bartholdi-n:rabbit}
Laurent Bartholdi and Volodymyr Nekrashevych.
\newblock Thurston equivalence of topological polynomials.
\newblock {\em Acta Math.}, 197:1--51, 2006.

\bibitem[B{\v{S}}01]{bartholdi-s:wpg}
Laurent Bartholdi and Zoran {\v{S}}uni{\'k}.
\newblock On the word and period growth of some groups of tree automorphisms.
\newblock {\em Comm. Algebra}, 29(11):4923--4964, 2001.

\bibitem[B{\v{S}}06]{bartholdi-s:bs}
Laurent Bartholdi and Zoran {\v{S}}uni{\'k}.
\newblock Some solvable automaton groups.
\newblock In {\em Topological and asymptotic aspects of group theory}, volume
  394 of {\em Contemp. Math.}, pages 11--29. Amer. Math. Soc., Providence, RI,
  2006.

\bibitem[BS07]{berlatto-s:virtual}
Adilson Berlatto and Said Sidki.
\newblock Virtual endomorphisms of nilpotent groups.
\newblock {\em Groups Geom. Dyn.}, 1(1):21--46, 2007.

\bibitem[BS10]{brunner-s:abelian}
Andrew~M. Brunner and Said~N. Sidki.
\newblock Abelian state-closed subgroups of automorphisms of $m$-ary trees.
\newblock {\em Groups Geom. Dyn.}, 4(3):455--472, 2010.

\bibitem[BV05]{bartholdi-v:basilica}
Laurent Bartholdi and B{\'a}lint Vir{\'a}g.
\newblock Amenability via random walks.
\newblock {\em Duke Math. J.}, 130(1):39--56, 2005.

\bibitem[GNS00]{grigorchuk-n-s:automata}
R.~I. Grigorchuk, V.~V. Nekrashevich, and V.~I. Sushchanski{\u\i}.
\newblock Automata, dynamical systems, and groups.
\newblock {\em Tr. Mat. Inst. Steklova}, 231(Din. Sist., Avtom. i Beskon.
  Gruppy):134--214, 2000.

\bibitem[Gri80]{grigorchuk:burnside}
R.~I. Grigorchuk.
\newblock On {B}urnside's problem on periodic groups.
\newblock {\em Funktsional. Anal. i Prilozhen.}, 14(1):53--54, 1980.

\bibitem[Gri84]{grigorchuk:gdegree}
R.~I. Grigorchuk.
\newblock Degrees of growth of finitely generated groups and the theory of
  invariant means.
\newblock {\em Izv. Akad. Nauk SSSR Ser. Mat.}, 48(5):939--985, 1984.

\bibitem[Gri85]{grigorchuk:pgps}
R.~I. Grigorchuk.
\newblock Degrees of growth of {$p$}-groups and torsion-free groups.
\newblock {\em Mat. Sb. (N.S.)}, 126(168)(2):194--214, 286, 1985.

\bibitem[GS83a]{gupta-s:burnside}
Narain~D. Gupta and Said~N. Sidki.
\newblock On the {B}urnside problem for periodic groups.
\newblock {\em Math. Z.}, 182(3):385--388, 1983.

\bibitem[GS83b]{gupta-s:p}
Narain~D. Gupta and Said~N. Sidki.
\newblock Some infinite $p$-groups.
\newblock {\em Algebra i Logika}, 22(5):584--589, 1983.

\bibitem[G{\v{S}}06]{grigorchuk-s:hanoi-cr}
Rostislav Grigorchuk and Zoran {\v{S}}uni{\'k}.
\newblock Asymptotic aspects of {S}chreier graphs and {H}anoi {T}owers groups.
\newblock {\em C. R. Math. Acad. Sci. Paris}, 342(8):545--550, 2006.

\bibitem[G{\v{S}}08]{grigorchuk-s:hanoi-spectrum}
Rostislav Grigorchuk and Zoran {\v{S}}uni{\'c}.
\newblock Schreier spectrum of the {H}anoi {T}owers group on three pegs.
\newblock In {\em Analysis on graphs and its applications}, volume~77 of {\em
  Proc. Sympos. Pure Math.}, pages 183--198. Amer. Math. Soc., Providence, RI,
  2008.

\bibitem[G{\.Z}02]{grigorchuk-z:basilica1}
Rostislav~I. Grigorchuk and Andrzej {\.Z}uk.
\newblock On a torsion-free weakly branch group defined by a three state
  automaton.
\newblock {\em Internat. J. Algebra Comput.}, 12(1-2):223--246, 2002.

\bibitem[Kal48]{kaloujnine:p-sylow}
L{\'e}o Kaloujnine.
\newblock La structure des {$p$}-groupes de {S}ylow des groupes sym\'etriques
  finis.
\newblock {\em Ann. Sci. \'Ecole Norm. Sup. (3)}, 65:239--276, 1948.

\bibitem[Kap08]{kapovich:arithmetic-selfsimilar}
Michael Kapovich.
\newblock Arithmetic aspects of self-similar groups.
\newblock arXiv:0809.0323, 2008.

\bibitem[Nek02]{nekrashevych:virtual}
Volodymyr Nekrashevych.
\newblock Virtual endomorphisms of groups.
\newblock {\em Algebra Discrete Math.}, (1):88--128, 2002.

\bibitem[Nek05]{nekrashevych:b-selfsimilar}
Volodymyr Nekrashevych.
\newblock {\em Self-similar groups}, volume 117 of {\em Mathematical Surveys
  and Monographs}.
\newblock American Mathematical Society, Providence, RI, 2005.

\bibitem[NS04]{nekrashevych-s:half}
V.~Nekrashevych and S.~Sidki.
\newblock Automorphisms of the binary tree: state-closed subgroups and dynamics
  of 1/2-endomorphisms.
\newblock In {\em Groups: topological, combinatorial and arithmetic aspects},
  volume 311 of {\em London Math. Soc. Lecture Note Ser.}, pages 375--404.
  Cambridge Univ. Press, Cambridge, 2004.

\bibitem[Sid00]{sidki:circuit}
Said Sidki.
\newblock Automorphisms of one-rooted trees: growth, circuit structure, and
  acyclicity.
\newblock {\em J. Math. Sci. (New York)}, 100(1):1925--1943, 2000.

\bibitem[{\v{S}}un07]{sunic:hausdorff}
Zoran {\v{S}}uni{\'c}.
\newblock Hausdorff dimension in a family of self-similar groups.
\newblock {\em Geom. Dedicata}, 124:213--236, 2007.

\bibitem[SV08]{savchuk-v:free-c2}
Dmytro Savchuk and Yaroslav Vorobets.
\newblock Automata generating free products of groups of order 2.
\newblock arXiv:0806.4801, 2008.

\bibitem[SVV06]{steinberg-v-v:free}
Benjamin Steinberg, Mariya Vorobets, and Yaroslav Vorobets.
\newblock Automata over a binary alphabet generating free groups of even rank.
\newblock arXiv:math/0610033, 2006.

\bibitem[VV07]{vorobets-v:free}
Mariya Vorobets and Yaroslav Vorobets.
\newblock On a free group of transformations defined by an automaton.
\newblock {\em Geom. Dedicata}, 124:237--249, 2007.

\bibitem[VV10]{vorobets-v:free2}
Mariya Vorobets and Yaroslav Vorobets.
\newblock On a series of finite automata defining free transformation groups.
\newblock {\em Groups Geom. Dyn.}, 4(2):377--405, 2010.

\end{thebibliography}

\end{document}